\documentclass[a4paper, 10pt]{amsart}

\usepackage{amsmath,amssymb}
\usepackage{xcolor}
\usepackage{mathabx}
\usepackage[margin = 3cm]{geometry}
\usepackage[english]{babel}
\usepackage{amsthm}
\usepackage{tikz-cd}
\usepackage[]{mathrsfs}
\usepackage{mathtools}
\usepackage{todonotes}
\usepackage{calc}
\usepackage{graphicx}
\usepackage{faktor}
\usepackage{amsfonts}
\usepackage{hyperref}
\hypersetup{colorlinks=true, linkcolor=black}
\usepackage{verbatim}
\usepackage{pinlabel}
\usepackage[T1]{fontenc}
\usepackage{graphicx}
\usepackage{graphbox}
\usepackage{caption}
\usepackage{subcaption}
\usepackage{enumitem}

\usepackage[backend=biber, style=numeric]{biblatex}
\DeclareFieldFormat[article, inbook, incollection, inproceedings, misc, online]{title}{\mkbibquote{#1\isdot}}

\theoremstyle{plain}
\newtheorem{theorem}{Theorem}[section]
\newtheorem{corollary}[theorem]{Corollary}
\newtheorem{lemma}[theorem]{Lemma}
\newtheorem{proposition}[theorem]{Proposition}

\newtheorem*{lemma*}{Lemma}
\newtheorem*{convention}{Convention}

\theoremstyle{definition}
\newtheorem{definition}[theorem]{Definition}

\theoremstyle{remark}

\newtheorem*{remark}{Remark}

\newtheorem*{question*}{Question}

\newtheorem*{defn*}{Definition}

\theoremstyle{remark}

\newtheorem*{rmk}{Remark}

\newcommand{\N}{\mathbb{N}}

\newcommand{\Z}{\mathbb{Z}}

\newcommand{\R}{\mathbb{R}}

\newcommand{\F}{\mathcal{F}}

\newcommand{\Ends}{\mathrm{Ends}}

\title{Left-orderability of Groups Acting on Bifoliated Planes}

\author[Mauro Camargo-Rios]{Mauro Camargo-Rios}
\address{Cornell University, Ithaca, NY}
\email{mcc326@cornell.edu}

\author[Lingfeng Lu]{Lingfeng Lu}
\address{Queen's University, Kingston, Ontario}
\email{21ll24@queensu.ca}

\begin{document}

\begin{abstract}
    We prove that any group acting faithfully on a bifoliated plane while preserving the orientations of both foliations is left-orderable. The proof utilizes a construction of a linear order on the set of ends of the leaf spaces, which takes advantage of the additional structure coming with a bifoliation. Moreover, we build an identification between ends of leaf spaces of the bifoliation and subsets of the boundary circle at infinity, and use it to give a condition for the equivalence between a faithful group action on a bifoliated plane and a faithful group action on the set of ends of the leaf spaces.
\end{abstract}

\maketitle

\section{Introduction}

A group is \emph{left-orderable} if it admits a linear order that is invariant under left multiplication. The definition is purely algebraic, but left-orderability appears frequently in topological and dynamical contexts. For example, one part of the \emph{L-space conjecture} by Boyer--Gordan--Watson associates the topology of a 3-manifold (in particular, the existence of a taut foliation on the manifold) with left-orderability of its fundamental group, whereas a link to dynamics was given by the classical result stating that a countable group is left-orderable if and only if it acts faithfully on $\R$ by orientation-preserving homeomorphisms. To interested readers, we recommend \cite{navas2010dynamics} for a dynamical approach to properties of left-orderable groups.

The other object of interest in this paper is a \emph{bifoliated plane}.
\begin{definition}
A \emph{bifoliated plane} $(P, \F_1, \F_2)$ consists of a topological plane $P$ equipped with $C^0$ foliations by lines $\F_1, \F_2$ such that:
    \begin{enumerate}
        \item $\F_1, \F_2$ are proper, that is, each leaf is properly embedded in $P$.
        \item $\F_1, \F_2$ are transverse.
    \end{enumerate}    
\end{definition}

From the definition, $\F_1$ and $\F_2$ are non-singular and orientable, and we will always assume that they are oriented. We prove the following result regarding left-orderability of groups acting on bifoliated planes:
 
\begin{theorem}\label{theorem: main}
    Let $P = (P, \F_1, \F_2)$ be a bifoliated plane. Let $\mathrm{Aut^+}(P, \F_1, \F_2) $ be the group of homeomorphisms of $P$ that preserves foliations $\F_1$ and $\F_2$ as well as their orientations. Then, $\mathrm{Aut}^+(P, \F_1, \F_2)$ is left-orderable.
    \end{theorem}
\begin{remark}
    For the sake of brevity, when the foliations $\F_1$ and $\F_2$ on the plane $P$ are clear from the context, we will refer to $\mathrm{Aut}^+(P, \F_1, \F_2)$ as $\mathrm{Aut}^+(P)$.
\end{remark}

As a consequence, since left-orderability of a group is inherited by its subgroups, one can see that any group acting faithfully on a bifoliated plane is left-orderable. Additionally, if the group is countable, then it can be realized as a subgroup of $\mathrm{Homeo^+}(\R)$.

\begin{remark}
    Theorem \ref{theorem: main} also implies that if $G$ is any group acting faithfully on a bifoliated plane, then it has a subgroup of index at most $4$ (which preserves orientations of both $\F_1$ and $\F_2$) that is left-orderable. 
\end{remark}

Foliations appear naturally and frequently in the study of dynamics, geometry and topology in three dimensions. For example, Fenley and Potrie (\cite{fenley2023intersection}, \cite{fenley2023transverse}) recently studied flows that are intersections of two transverse 2-dimensional foliations on closed 3-manifolds, and proved applications to partially hyperbolic diffeomorphisms when the orbit spaces of those flows are Hausdorff, i.e., bifoliated planes. To further motivate the study of bifoliated planes, consider the classical example: given an Anosov flow $\varphi_t$ on a $3$-manifold $M$, the orbit space $P_\varphi$ of its lift $\Tilde{\varphi_t}$ to the universal cover $\Tilde{M}$ of $M$ can be shown to be homeomorphic to $\R^2$ (\cite{fenley1994anosov}), and a pair of transverse foliations by lines, $\F^s$ and $\F^u$, is obtained from the weak stable and weak unstable foliations of the flow $\Tilde{\varphi}_t$. Thus, $(P_\varphi, \F^s, \F^u)$ is a bifoliated plane, which is additionally equipped with an action of $\pi_1(M)$ induced by the deck group action on $\tilde{M}$. It was shown in \cite{barthelme2022orbit} that, by studying the action of $\pi_1(M)$ on both $(P_\varphi, \F^+, \F^-)$, one can classify (pseudo-)Anosov flows up to orbit equivalence. This motivates the study of bifoliated planes as well as groups acting on them.

%Fenley \cite{fenley2012ideal} , which encodes much of the dynamical information about the flow. 
The proof of Theorem \ref{theorem: main} involves introducing an order on the set of \emph{ends}, which are equivalence classes of rays, in the leaf spaces of the bifoliation. This was inspired by Zhao's work in \cite{zhao}, where he showed that for a $3$-manifold $M$ admitting a taut foliation with \emph{orderable cataclysms}, the fundamental group $\pi_1(M)$ is left-orderable. In general, given an action of a group $G$ on a non-Hausdorff $1$-manifold $\Lambda$, it is not the case that $\Lambda$ always has orderable cataclysms with respect to this action (see for instance Example 3.7 in \cite{calegari2003laminations}). However, with the extra structure coming from a bifoliation, there is a natural way of defining an order on every cataclysm in each leaf space that is preserved by actions of $G$.

In the setting of general bifoliated planes (not necessarily coming from an Anosov flow), Bonatti (\cite{bonatti}) defined the canonical \emph{boundary circle at infinity} via ideas similar to the work of Mather in \cite{mather}. The boundary at infinity is homeomorphic to a circle and compactifies the plane into a closed disk. Fenley (\cite{fenley2012ideal}) had earlier defined the notion of circle at infinity for the bifoliated plane associated to an Anosov flow. Given an action by a group $G$ on a bifoliated plane, there is a natural induced action on its circle at infinity. In the case of \emph{Anosov-like} actions on bifoliated planes (which include the actions associated to an Anosov flow), the study of this induced action on the circle at infinity was a key ingredient in \cite{barthelme2022orbit}.

Since rays of leaves limit onto points of the boundary circle, it is natural to ask whether rays in the leaf spaces of the foliations limit to points on the boundary circle in a meaningful way. We answer this question, describing a natural correspondence between the set of ends of leaf spaces of a bifoliated plane and a subset of the associated boundary circle at infinity $\partial P$ -- we call this the \emph{realization} of ends. We will show that there are two \emph{types} of realizations: the point-type and the interval-type. The point-type realization comes from a sequence of leaves with their \emph{ideal points} converging to the same limit, while the interval-type realization comes from a local bifoliation structure called the \emph{infinite product region}. One application of this correspondence is that under certain conditions, we may obtain a faithful action on the bifoliated plane from a faithful action on the set of ends, which a priori asks for less information:

\begin{theorem}\label{theorem: faithful}
    Let $G$ be a group that acts on $P = (P, \F_1, \F_2)$ by homeomorphisms, preserving foliations and their orientations. The action by $G$ on $P$ is faithful if and only if the induced action on $\partial P$ is faithful. Moreover, if the set of point-type realizations of ends is dense in $\partial P$, then the action on $P$ is faithful if and only if the induced action on ends in leaf spaces is faithful.
\end{theorem}

Similar to the result for groups acting on $\R$ and left-orderability, another classical result states that a group is circularly orderable if it acts faithfully on $\mathbb{S}^1$ by orientation-preserving homeomorphisms. Under no additional assumption, circularly orderable groups are not always left-orderable, since the action on $\mathbb{S}^1$ does not always lift to an action on $\mathbb{R}$. In \cite{bell2021promoting}, Bell--Clay--Ghaswala gives an algebraic criterion for promoting circular orderability to left-orderability: a group $G$ is left-orderable if and only if $G \times \Z \addslash n\Z$ is circularly orderable for $n > 1$. However, with Theorem \ref{theorem: main} and the first part of Theorem \ref{theorem: faithful}, we know that the condition is met by groups acting faithfully (preserving foliations and orientations) on a bifoliated plane. \par

\subsection*{Outline} This paper is structured in the following way: In Section \ref{section: prelim}, we give necessary background on 1-manifolds, discuss Zhao's construction (Subsection \ref{subsection: zhao}) and extend it to the setting of bifoliated planes (Subsection \ref{subsection: bifoliated planes}); The proof of Theorem \ref{theorem: main} is presented in Section \ref{section: main}; In Section \ref{section: circle at infinity}, we start by defining the boundary circle at infinity, then we describe behaviors and structures of various objects on a bifoliated plane, construct the aforementioned correspondence, and prove Theorem \ref{theorem: faithful}.

\subsection*{Acknowledgements}
The authors are grateful to Thomas Barthelm\'{e} and Kathryn Mann for their guidance and many valuable comments.

\section{Preliminaries}\label{section: prelim}

\subsection{1-manifolds}\label{subsection: zhao}

Let $\Lambda$ be a non-Hausdorff and simply connected $1$-manifold, equipped with an action by a group $G$. Since simply connected manifolds are orientable, we will also assume $\Lambda$ to be oriented. We start by briefly stating some results concerning $\Lambda$. Most importantly, we will later make use of the fact that if $\Lambda$ has \emph{orderable cataclysms} with respect to the action of $G$, there exists a $G$-invariant linear order on $\mathrm{Ends}(\Lambda)$, the space of ends of $\Lambda$. This is a result of Zhao in \cite{zhao}.

We begin by stating the definition of (orderable) cataclysms. Recall that two points in a topological space are said to be \emph{non-separated} if any neighborhood of one of them intersects every neighborhood of the other.

\begin{definition}
   A \emph{cataclysm} in $\Lambda$ is a maximal collection of pairwise non-separated points. 

   Given an action of a group $G$ on $\Lambda$ by homeomorphisms, we say that $\Lambda$ has \emph{orderable cataclysms with respect to the action of} $G$ if there exists a $G$-invariant linear order in each cataclysm $\mu$ of $\Lambda$.
\end{definition}

This definition was made originally by Calegari and Dunfield in \cite{calegari2003laminations}, in the more general context of laminations. \par

It is a well known fact that a group acting faithfully on a linearly ordered set in an order-preserving way must be left-orderable. In \cite{zhao}, Zhao shows that given a simply connected and orientable 1-manifold $\Lambda$ having orderable cataclysms with respect to an action by a group $G$, there exists one such faithful order-preserving action for $G$. The linearly ordered set on which $G$ acts in an order-preserving way is the set of \emph{positive ends} of $\mathrm{\Lambda}$. The distinction between positive and negative ends of $\Lambda$ is possible since $\Lambda$ is assumed to be oriented.

\begin{definition}\label{def: rays and ends}
    A \emph{ray} in $\Lambda$ is a proper embedding $r: [0, +\infty) \to \Lambda$. Let $\mathcal{E}$ be the set of rays in $\Lambda$, and let $\sim$ be the equivalence relation on $\mathcal{E}$ such that two rays are equivalent if and only if they can be reparametrized so that they agree after a certain point. Then
    $$\Ends(\Lambda) \coloneqq \mathcal{E}/\sim$$
    is the set of \emph{ends} of $\Lambda$. An end is said to be \emph{positive} if a ray (hence all rays) representing it is orientation preserving, where we give $[0, +\infty)$ the standard orientation. Otherwise, the end is said to be \emph{negative}. Sets of positive and negative ends are denoted by $\Ends_+(\Lambda)$ and $\Ends_-(\Lambda)$, respectively.
\end{definition}
It is clear from the definition of $\mathrm{Ends}(\Lambda)$ that an action of $G$ on $\Lambda$ by homeomorphisms induces an action of $G$ on $\mathrm{Ends}(\Lambda)$. Moreover, if the original action is orientation preserving, then the set of positive ends is invariant under this action.

In the proof of our main theorem, we will make use of the following result.

\begin{proposition}[\cite{zhao}]\label{order_zhao}
    If $\Lambda$ has orderable cataclysms with respect to the action of a group $G$, then there exists a left-order on the set $\Ends_+(\Lambda)$ that is preserved by the induced action of $G$.
\end{proposition}

Using the above proposition together with 3-manifold topology, Zhao deduced left-orderability of certain 3-manifold groups:

\begin{theorem}[\cite{zhao}]
     Let $M$ be a closed, connected, orientable, irreducible $3$-manifold. If $M$ admits a taut foliation $\F$ with orderable cataclysm, then $\pi_1(M)$ is left-orderable.
\end{theorem}

We briefly explain how Zhao defines the left-order on $\mathrm{Ends}_+(\Lambda)$ mentioned in Proposition \ref{order_zhao}.

Given two distinct points $u, v$ in $\Lambda$, since $\Lambda$ may not be Hausdorff, it is possible that there does not exist an embedded path joining $u$ and $v$. However, for any $u,v$ in $\Lambda$ there exists a unique \emph{broken path} from $u$ to $v$, which consists of a sequence of embedded closed intervals $[a_0, a_1], [a_2, a_3], \dots, [a_{2n}, a_{2n+1}]$ such that $a_0 = u, a_{2n+1} = v$, and where $a_{2i - 1}, a_{2i }$ are contained in the same cataclysm (see \cite{calegari2007foliations}, Subsection 4.3 for details). Each pair of points $(a_{2i - 1}, a_{2i })$ is called a \emph{cusp} and can be either \emph{positive} or \emph{negative}, depending on how points in the pair are ordered within the cataclysm. \par

A broken path between two points of $\Lambda$ can be extended to a broken path between distinct ends. For any distinct positive ends $x_1$, $x_2$ in $\mathrm{Ends}^+(\Lambda)$, Zhao then defined the quantity $n(x_1, x_2)$ to be the difference between the number of positive cusps and the number of negative cusps on the broken path from $x_1$ to $x_2$. It is then shown that $n(x_1, x_2)$ possesses all necessary properties to define a linear order $ \overset{n}{<}$ on $\mathrm{Ends}_+(\Lambda)$ as follows: given $x_1, x_2 $ in $\mathrm{Ends}_+(\Lambda)$,  $x_1 \overset{n}{<} x_2$ if and only if $n(x_1, x_2) < 0$. Finally, it is easy to see from the definition of $n(x_1, x_2)$ that, since $G$ acts on $\Lambda$ by orientation preserving homeomorphisms and $\Lambda$ has orderable cataclysms with respect to this action, this order is $G$-invariant.

\subsection{Bifoliated planes}\label{subsection: bifoliated planes}

Let $P = (P, \F_1, \F_2)$ be a bifoliated plane. Recall that the leaf space of a foliation $\F_i$ is defined to be $\Lambda_i = P / \F_i$, that is the space obtained by identifying each leaf of the foliation to a point, equipped with the quotient topology. We use $q_i \colon P \to \Lambda_i$ to denote the projection map from $P$ to the leaf space $\Lambda_i$ of the foliation $\F_i$. The following is a classical result restated for a bifoliated plane, readers may refer to Appendix D in \cite{candel2000foliations} for the original statement.

\begin{proposition}[\cite{candel2000foliations}]\label{proposition : leaf_spaces}
    If $(P, \F_1, \F_2)$ is a bifoliated plane, then the leaf spaces $\Lambda_1, \Lambda_2$ of $\F_1, \F_2$ respectively, are simply connected and second countable 1-manifolds. Additionally, once one has chosen orientations for $\F_1$ and $\F_2$, these orientations induce orientations on $\Lambda_2$ and $\Lambda_1$, respectively. 
\end{proposition}

This holds in fact for the leaf space of any proper foliation by lines in $\R^2$, since one only needs existence of local transversals to the foliation in order to show that these leaf spaces are locally Euclidean. Note that the leaf spaces are Hausdorff if and only if they are homeomorphic to $\R$.\par

In this section and in the rest of this paper, 1-manifolds are understood to be second countable (but not necessarily Hausdorff), connected, and simply connected. We also assume that they are oriented whenever it is required. \par

We will first present some definitions and structures on bifoliated planes, as well as some examples, and then build a linear order on the set of all ends -- both positive and negative -- of the leaf spaces of a bifoliated plane.

\begin{rmk}
    Given a bifoliated plane $(P, \F_1, \F_2)$ equipped with an action of a group $G$ preserving the orientations of both foliations, there is an induced action on $\mathrm{Ends}(\Lambda_i) = \mathrm{Ends}_+(\Lambda_i) \cup \mathrm{Ends}_-(\Lambda_i) $ for $i=1,2$, which preserves the positive and negative ends.
\end{rmk}

The following definition describes how leaves of a foliation look in the bifoliated plane when its leaf space is non-Hausdorff.

\begin{definition}
    Two leaves of a foliation $\mathcal{F}$ are said to be \emph{non-separated} if they do not admit disjoint neighborhoods saturated by leaves of $\mathcal{F}$, or equivalently, if their projections in the leaf space are contained in the same cataclysm.
\end{definition}

A saturated neighborhood of a leaf $l \in \F_i$ is given by leaves in $\F_i$ that intersect a transversal to $l$. This implies that two leaves of $\mathcal{F}$ are non-separated if and only if there exists a sequence of leaves of $\mathcal{F}$ that accumulates on both of them. Given $l_1, l_2, l_3 \in \F_i$, if $l_1$ is non-separated from $l_2$ and $l_2$ is non-separated from $l_3$, it is not necessary for $l_1$ and $l_3$ to be non-separated (see Figure \ref{fig: non-separated leaves}).

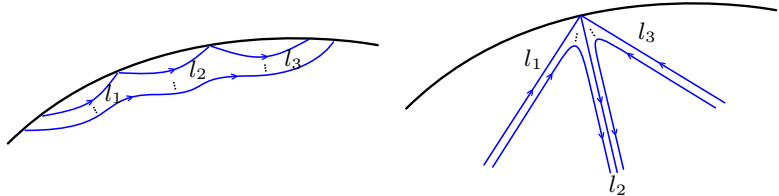
\begin{figure}[h]
    \labellist
    \small\hair 2pt
    \pinlabel $l_1$ at 80 60
    \pinlabel $l_2$ at 143 78
    \pinlabel $l_3$ at 214 86
    \pinlabel $l_1$ at 393 86
    \pinlabel $l_2$ at 456 -9
    \pinlabel $l_3$ at 478 106
     \endlabellist
    \centering
    \vspace{-2mm}\includegraphics[scale = 0.50]{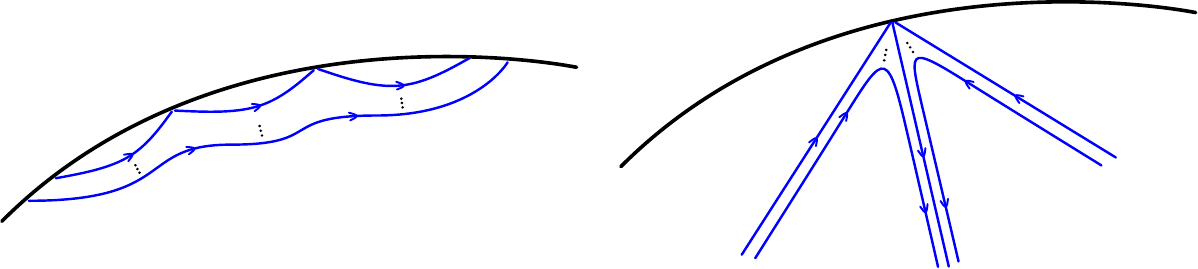}
    \caption{In the left figure, $l_1$, $l_2$, and $l_3$ are pairwise non-separated; in the right figure, $l_1$ and $l_2$ are non-separated, $l_2$ and $l_3$ are non-separated, while $l_1$ and $l_3$ are separated by $l_2$.}
    \label{fig: non-separated leaves}
\end{figure}    

\begin{comment}
    
\begin{definition}
    Let $l \in \F_1$ and $f \in \F_2$. We say that $l$ and $f$ make a \emph{perfect fit} if there exists a transversal $\tau_1$ to $\F_1$ starting at a point on $l$ and a transversal $\tau_2$ to $\F_2$ starting at a point on $f$ such that
    \begin{itemize}
        \item every $s \in \F_1$ that intersects $\tau_1$ also intersects $f$, and
        \item every $t \in \F_2$ that intersects $\tau_2$ also intersects $l$.
    \end{itemize}
\end{definition}

\begin{figure}[h]
    \centering
    \labellist
    \small\hair 2pt
    \pinlabel $l$ at 44 90
    \pinlabel $f$ at 206 23
    \pinlabel $\tau_1$ at 128 74
    \pinlabel $\tau_2$ at 156 68
    \pinlabel $s$ at 27 56
    \pinlabel $t$ at 173 17
    \endlabellist
    \includegraphics[scale = 0.60]{perfect-fit}
    \caption{A perfect fit made by $l$ and $f$}
    \label{fig: perfect fit}
\end{figure}

\end{comment}

As mentioned in the introduction, an important class of examples of bifoliated planes equipped with a group action comes from Anosov flows on $3$-manifolds. In that case, given an Anosov flow on a 3-manifold $M$, there is an action by the fundamental group $\pi_1(M)$ induced by the deck group action of $\pi_1(M)$ on the universal cover of $M$.

%The following theorem was proved independently by Barbot (\cite{barbot_1995}) and Fenley (\cite{fenley1994anosov}):

%\begin{theorem}
%Let $M= M^3$, and let $\varphi_t: M \to M$ be an Anosov flow. Then, the orbit space $P_\varphi$ of its lift $\tilde{\varphi}_t : \tilde{M} \to \tilde{M}$ is a topological plane.
%\end{theorem}

%Since the leaves of the weak stable and weak unstable foliations $\tilde{W}^{s}, \tilde{W}^u $ of $\tilde{\varphi}_t$ are properly embedded in $\tilde{M}$ and are topological planes foliated by orbits, they descend to a pair of one-dimensional, proper foliations $\F^s, \F^u$ on $P_\varphi$. Note that any two leaves $\tilde{W}_x^s, \tilde{W}_y^u$ intersect if and only if there exists $z\in \tilde{M}$ such that $\tilde{W}_x^s = \tilde{W}_z^s, \tilde{W}_y^u = \tilde{W}_z^u $. If that is the case, then the fact that the flow is Anosov implies that $\tilde{W}_x^s $ and $ \tilde{W}_y^u$ intersect transversely and exactly along the orbit of $z$. This means that any two leaves of $\F^s, \F^u$ are either disjoint or intersect transversely in a single point of $P_\varphi$. Therefore, $(P_\varphi, \F^s, \F^u)$ is a bifoliated plane. Moreover, the fundamental group $\pi_1(M)$ of $M$ acts on $\tilde{M}$ by deck transformation, which takes orbits to orbits and preserves foliations, so it acts on $(P_\varphi, \F^s, \F^u)$.

\begin{rmk}
    
For any bifoliated plane $(P_\varphi, \F^s, \F^u)$ coming from an Anosov flow $\varphi_t$, the leaf spaces $\Lambda^s, \Lambda^u$ of $\F^s, \F^u$ are either both Hausdorff (and hence homeomorphic to $\R$) or both non-Hausdorff. This was shown by Barbot and Fenley (see \cite{barbot_1995}, \cite{fenley1995sided}). There exist examples of both behaviors. In a general bifoliated plane this is not necessarily the case, and it is not assumed in this paper.
\end{rmk}

\section{Proof of Theorem \ref{theorem: main}}\label{section: main}
 
Given a bifoliated plane $P = (P, \F_1, \F_2)$ and a group $G \subseteq \mathrm{Aut}^+(P)$, there is an induced action of $G$ on the leaf spaces $\Lambda_1, \Lambda_2$. The first goal of this section is to show that the $\Lambda_i$ have orderable cataclysms with respect to this action, and then to use this to define a $G$-invariant linear order on the set $\mathrm{Ends}(\Lambda_1) \cup \mathrm{Ends}(\Lambda_2)$. We begin by explicitly defining the order on each cataclysm, and we then show that it is indeed a linear order. 

The first step is to define a partial order on the leaf spaces $\Lambda_i$. By Proposition \ref{proposition : leaf_spaces}, fixing a continuous orientation on leaves of each of $\F_1$ and $\F_2$ induces an orientation on $\Lambda_2$ and $\Lambda_1$, respectively. Orientability of each $\Lambda_i$ allows us to define a partial order on $\Lambda_i$ as follows:

\begin{definition}
    Let $x,y $ in $\Lambda_i$. We say that $y > x$ if there exists an orientation preserving embedding $\gamma\colon [0,1] \to \Lambda_i$ such that $\gamma(0) = x, \gamma(1) = y$.
\end{definition}
It is easy to check that this relation is a partial order on $\Lambda_i$.

Now, note that the complement of any leaf of each $\F_i$ has exactly two connected components. We will use this fact to define the order on cataclysms.

\begin{definition}
    Let $l$ be a leaf of $\F_1$ (resp. $\F_2$). We define the \emph{front of } $l$ to be the connected component of $P\setminus l$ that consists of all the leaves $l'$ of $\F_1$ such that $l' > l$ in $\Lambda_1$.

    Likewise, we define the \emph{back of} $l$ to be the connected component of $P\setminus l$ containing some leaf $l''$ of $\F_1$ such that $l'' < l$.
\end{definition}

\begin{remark}
     The front and the back of $l$ are disjoint, and therefore their union is $P\setminus l$. Thus, given two leaves $l, l'$ of $\F_1$ (resp. $\F_2$), one has that $l'$ is either in the back or in the front of $l$. 
     
     However, note that (perhaps counterintuitively) it is possible to have $l$ and $l'$ be in the front (or the back) of each other. That is, it is not true that the fact that $l'$ is in front (back) of $l$ implies that $l$ is in the back (front) of $l'$. For an example, see Figure \ref{fig: lemma_nonsep_figure}: the leaves $l_1, l_2$ are in the back of each other. However, the next lemma shows that this does not happen when the leaves intersect a pair of non-separated leaves (a condition satisfied by leaves $s_1, s_2$ in Figure \ref{fig: lemma_nonsep_figure}).

\end{remark}

%\begin{definition}
%    Let $l$ be a leaf of $\F_1$ (resp. $\F_2$), and let $s$ be a leaf of $\F_2$ (resp. $\F_1$) that intersects it. 
%    Then, given an orientation preserving parametrization $\alpha\colon \R \to s$ such that $\alpha(0) = s\cap l$, the $\emph{front}$ of $l$ is the connected component of $P\setminus l$ that contains $\alpha(\R^+)$, and we call the other component of $P\setminus l$ the \emph{back} of $l$.
%\end{definition}

\begin{lemma}\label{non_sep}
Let $l_1, l_2$ be non-separated leaves of $\F_1$. Then, there exists leaves $s_1, s_2 $ of $ \F_2$ such that $s_i \cap l_{i} \neq \emptyset$ and such that either $s_1$ is in the back of $s_2$ or $s_2$ is in the back of $s_1$. Moreover, if $s_1$ is in the back of $s_2$ (or vice-versa), then this also holds for any other choice of such leaves $s_1', s_2'$.
\end{lemma}

\begin{proof}[Proof of Lemma \ref{non_sep}]

Let $(l_n')_{n\in \mathbb{N}}$ be a sequence of leaves of $\F_1$ that converges to a set of leaves containing $l_1$ and $l_2$, and let $U$ denote the region bounded by $l_1$ and $l_2$. First, we show the following:

\begin{lemma*}
The leaves $l_1$ and $l_2$ are oriented so that their orientation coincides with the boundary orientation induced by an orientation of $U$.
\end{lemma*}

\begin{proof}
Consider trivially foliated (and trivially oriented), disjoint rectangular neighborhoods $R_1$ and $R_2$ intersecting $l_1$ and $l_2$ respectively. Taking $N >0$ large enough, we have that $l_n'$ intersects both $R_1$ and $R_2$ for all $n \geq N$, and taking a subsequence we may assume that the $l_n'$ for $n\geq N$ are in the same connected component of $U \setminus l_N'$. Note that since the foliations are assumed to be nonsingular, any leaf can intersect $R_i$ at most once, and all leaves must have their first intersection with $R_i$ on the same side of $R_i$, for $i=1,2$. The notion of ``first intersection'' of a leaf with $R_i$ is well defined since the foliation is oriented.

Now, suppose that $l_1$ and $l_2$ are not oriented in such a way that their orientation coincides with the boundary orientation corresponding to an orientation of $U$ (see Figure \ref{fig: lemma_nonsep_impossible}, on the left). Then, $l_1$ and $l_2$ are in different connected components of $U \setminus l_N'$. This implies that the leaves $l_n'$ for $n>N$ cannot accumulate on one of $l_1$ and $l_2$, which is a contradiction. This proves the lemma.
\end{proof}

\begin{figure}[h]
  \centering
    \includegraphics[width=0.80\textwidth]{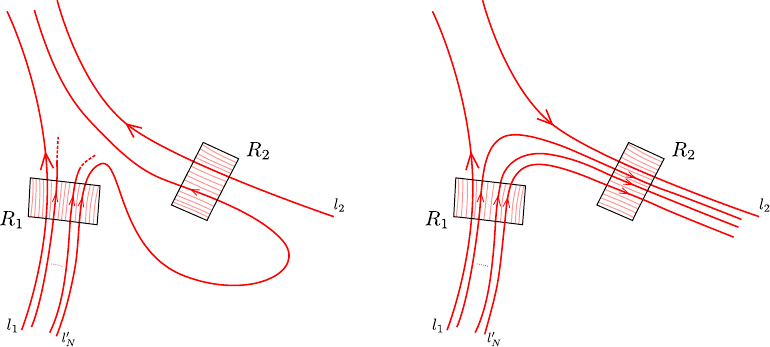}
    \caption{Two ways of orienting $l_1$ and $l_2$. The orientations in the picture on the left are not compatible with $l_1$ and $l_2$ being non-separated.}
    \label{fig: lemma_nonsep_impossible}
\end{figure}

Due to the lemma shown above we may then assume without loss of generality that $l_1, l_2$ are oriented as in Figure \ref{fig: lemma_nonsep_figure}.

\begin{figure}[h]
  \centering
    \includegraphics[width=0.40\textwidth]{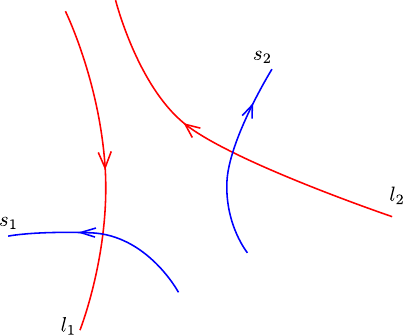}
    \caption{Two non-separated leaves $l_1, l_2$.}
    \label{fig: lemma_nonsep_figure}
\end{figure}

Let $s_1$, $s_2$ be leaves of $\F_2$ such that $s_i \cap l_i \neq \emptyset$. These leaves have orientations coming from the orientation on the foliation $\F_2$. There are two possibilities: either both of the positively oriented rays contained in the $s_i$ and based at $s_i \cap l_i$ are contained in the region between $l_1$ and $l_2$, or the same is true of the analogous negatively oriented rays (no other cases are possible, since the foliations $\F_1, \F_2$ must have the same local product orientation at $s_1 \cap l_1$ and $s_2 \cap l_2$). 

In Figure \ref{fig: lemma_nonsep_figure} we depict the second case. In this case, we can see that $s_2$ is in the back of $s_1$, and $s_1$ is not in the back of $s_2$. The same is true in the other case: recall that it is the orientation of $l_i$ that determines which component of $P \setminus s_1$ is the front of $s_i$. Therefore, reversing the orientations of both of the $s_i$ while leaving those of the $l_i$ unchanged does not change what the front (or back) of each leaf $s_i$ is.

%this is determined by the orientation of $\F_1$ (which is the transverse orientation of $\F_2$).

Moreover, any other choice of $s_1', s_2'$ such that $s_i' \cap l_i \neq \emptyset$ can be obtained by continuously varying $s_1$ and $s_2$, so that $s_i$ is in the back of $s_j$ if and only if $s_i'$ is in the back of $s_j'$, for $i,j \in \{ 1,2\}, i\neq j$.
\end{proof}

\begin{definition}\label{order}
    Given two non-separated leaves $l_1, l_2$ of $\F_1$, let $s_1, s_2$ be leaves of $\F_2$ intersecting them. Then, we say that $l_1 < l_2$ (\emph{resp.} $l_1 > l_2$) if $s_1$ is in the back (\emph{resp.} front) of $s_2$.
\end{definition}

%The definition above allows us to order any collection of pairwise non-separated leaves, which in turn allows us to order corresponding cataclysms in the leaf spaces $\Lambda_i, i= 1,2$. Moreover, orders defined in this way are preserved under any group action that preserves the orientations of the foliations.

We now show that this gives a total order on each cataclysm, and these orders are preserved by the action of $\mathrm{Aut}^+(P)$.

%\begin{definition} already defined earlier
%    Let $(P, \F_1, \F_2)$ be a bifoliated plane, where $\F_1$ and $\F_2$ are oriented. We define $\mathrm{Aut}^+(P, \F_1, \F_2)$ to be the group of homeomorphisms of $P$ that preserves the foliations $\F_1$ and $\F_2$, together with their orientations.

%    For brevity, we write $\mathrm{Aut}^+(P)$ instead of $\mathrm{Aut}^+(P,\F_1, \F_2)$ whenever the specific foliations $\F_1, \F_2$ are implied within the given context.  
%\end{definition}

\begin{lemma}\label{orderable_cataclysms}
    Each leaf space $\Lambda_i, i=1,2$ has orderable cataclysms with respect to the action of $\mathrm{Aut}^+(P)$. In particular, this is also true for any subgroup $G \subset \mathrm{Aut}^+(P)$.
\end{lemma}

\begin{proof}
   Let $\mu$ be a cataclysm in $\Lambda_1$. Then $\mu$ is a collection of non-separated $\F_1$-leaves. For any two leaves $l_1, l_2 \in \mu$, the relation described in Definition \ref{order} is well defined by Lemma \ref{non_sep} and is preserved by the action of $\mathrm{Aut}^+(P)$. We need to check that this relation is transitive in order to show that it defines a linear order on $\mu$. Suppose $l_1 < l_2 $ and $l_2 < l_3$, and we want to show that $l_1 < l_3$. 
\begin{figure}[h!]
  \centering
    \includegraphics[width=0.99\textwidth]{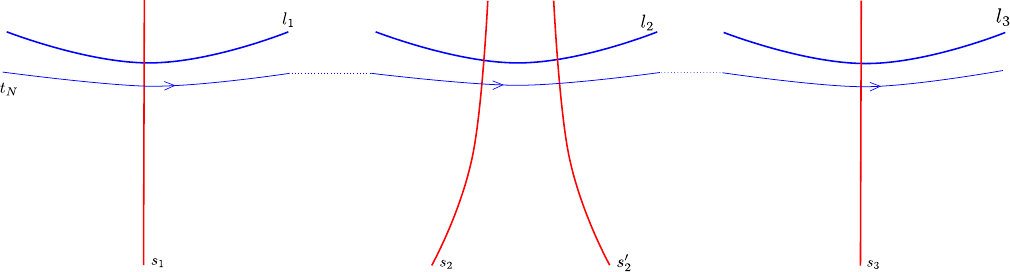}
    \caption{Transitivity of the order on a cataclysm}
    \label{lemma_cataclysms}
\end{figure}
   Since $l_1 < l_2$, there exist $s_1, s_2$ leaves of $\F_2$ intersecting $l_1, l_2$ respectively, such that $s_1$ is in the back of $s_2$. Analogously, there exist $s_2', s_3$ leaves of $\F_2$ intersecting $l_2, l_3$ respectively, such that $s_2'$ is in the back of $s_3$. 
   Let $(t_n)_{n\geq 1}$ be a sequence of $\F_1$-leaves that converges to (possibly among other leaves) $l_1, l_2, l_3$ in $\Lambda_1$. In the plane $P$, for $N$ large enough we must then have that $t_N$ intersects $s_1, s_2, s_2', s_3$. Additionally, the leaves $l_1, l_2, l_3$ must be on the same side of the leaf $t_N$, since otherwise there could not exist a sequence converging to all of them. Therefore, the situation is as shown in Figure \ref{lemma_cataclysms}.
      
   The fact that $s_1, s_2, s_2'$ and $s_3$ intersect $t_N$ together with the fact that $s_1$ is in the back of $s_2$ and $s_2'$ is in the back of $s_3$ implies that (regardless of whether $s_2$ is in the back of $s_2'$ or viceversa) $s_1$ is in the back of $s_3$, meaning that $l_1 < l_3$. This shows that the relation $<$ is transitive, and therefore defines a linear order on $\mu$.\end{proof}

Finally, we can define the desired $\mathrm{Aut}^+(P)$-invariant order:

\begin{proposition}\label{proposition: order_on_ends}
    Let $(P, \F_1, \F_2)$ be a bifoliated plane. Then, there exists an $\mathrm{Aut}^+(P)$-invariant linear order on $\mathrm{Ends}(\Lambda_1) \cup \mathrm{Ends}(\Lambda_2)$.

\end{proposition}

\begin{proof}

First, let $<_1^+$ be an $\mathrm{Aut}^+(P)-$invariant linear order on $\mathrm{Ends}_+(\Lambda_1)$, provided by Proposition \ref{order_zhao} and Lemma \ref{orderable_cataclysms}. Using Proposition \ref{order_zhao} again but with orientations reversed yields an $\mathrm{Aut}^+(P)$-invariant linear order $<_1^-$ on $\mathrm{Ends}_-(\Lambda_1)$. 

Since $\mathrm{Ends}(\Lambda_1) = \mathrm{Ends}_+(\Lambda_1) \sqcup \mathrm{Ends}_-(\Lambda_1)$ and these sets are $\mathrm{Aut}^+(P)$-invariant (the action of $ \mathrm{Aut}^+(P)$ preserves the orientation of both leaf spaces), we can define an order $<_1$ on $\mathrm{Ends}(\Lambda_1)$ by declaring $x <_1 y$ for all $x$ in $\mathrm{Ends}_+(\Lambda_1)$ and all $y $ in $\mathrm{Ends}_-(\Lambda_1)$. When $x,y$ are both in $\mathrm{Ends}_+(\Lambda_1)$ (resp. $\mathrm{Ends}_-(\Lambda_1)$), we can let $x<_1 y$ if and only if $x <_1^+ y$ (resp. $x <_1^- y$).

This is by definition an $\mathrm{Aut}^+(P)$-invariant linear order on $\mathrm{Ends}(\Lambda_1)$. The same procedure shows the existence of an analogous invariant linear order $<_2$ on $\mathrm{Ends}(\Lambda_2)$. Finally, the $\mathrm{Aut}^+(P)$-invariant linear orders on $\mathrm{Ends}(\Lambda_1)$ and $\mathrm{Ends}(\Lambda_2)$ can be combined to yield an $\mathrm{Aut}^+(P)$-invariant order $<$ on $\mathrm{Ends}(\Lambda_1) \cup \mathrm{Ends}(\Lambda_2)$, as we wanted: let  $x < y $ for all $x \in \mathrm{Ends}(\Lambda_1), y \in \mathrm{Ends}(\Lambda_2)$. For $x, y \in \mathrm{Ends}(\Lambda_i)$ and $i = 1,2$, let $x < y$ if and only if $x <_i y$.
\end{proof}

We now use this order to prove our main theorem.

\begin{proof}[Proof of Theorem \ref{theorem: main}]

We will use an idea of Zhao in section 3.3 of \cite{zhao}, with some modifications. The main step of this idea is to show that the subgroup of $\mathrm{Aut}^+(P)$ acting trivially on the ends of both leaf spaces is left-orderable.

Let $P = (P, \F_1, \F_2)$ be a bifoliated plane, and let $G = \mathrm{Aut}^+(P)$.
%Given a bifoliated plane $(P, \F_1, \F_2)$, by Lemma \ref{orderable_cataclysms} we know that every cataclysm can be assigned a linear order preserved by the action of $\mathrm{Aut}^+(P)$. Therefore, we are in a position to apply Proposition \ref{proposition: order_on_ends}.

By Proposition \ref{proposition: order_on_ends}, there exists a $G$-invariant linear order on $\mathrm{Ends}(\Lambda_1) \cup \mathrm{Ends}(\Lambda_2) $. Let
\begin{align*}
    H&= \{ g\in G : g\cdot x = x \text{ for all } x\in \mathrm{Ends}(\Lambda_1) \cup \mathrm{Ends}(\Lambda_2) \}.
\end{align*}
Then $G/H$ by definition acts faithfully on $\mathrm{Ends}(\Lambda_1) \cup \mathrm{Ends}(\Lambda_2)$ and preserves the linear order on it. This gives a left-order on $G/H$: by fixing a basepoint $x_0 \in \mathrm{Ends}(\Lambda_1) \cup \mathrm{Ends}(\Lambda_2)$, one can define $gH <_{G/H} g'H$ if and only if $g\cdot x_0 < g'\cdot x_0$ (for a complete proof, see Example III of Section 5 in \cite{conrad}).

Having shown the existence of a left-order on $G/H$, the following classical result reduces the proof of Theorem \ref{theorem: main} to showing the existence of a left-order on $H$.

\begin{lemma*}[\cite{conrad}, Section 3.7]
    If $<_{G/H}$ is a left-order on $G/H$ and $<_H$ is a left-order on $H$, then a left-order can be defined on $G$ as follows: given $g,g' $ in $G$, we say $g'<g$ if  one of two conditions holds: either $g'H <_{G/H} gH$, or $g'H = gH$ and $g^{-1}g' <_H e$.
\end{lemma*}

%, defined roughly in the same way as the lexicographic order on a product of ordered sets: we say that $ g' < g $ if . 
Now, we show that there indeed exists a  left-order on $H$.

By second-countability of the leaf spaces, there exists a countable dense subset $X = \{x_i\}_{i \in \mathcal{I}} \subset \Lambda_1$. For each $i \in \mathcal{I}$, let $r^+_i$ and $r^-_i$ be a positive ray and a negative ray based at $x_i$, respectively. Let $R_i = r^+_i \cup r^-_i$. Then $R_i$ is homeomorphic to $\R$ for all $i \in \mathcal{I}$, so the collection $\{R_i\}_{i \in \mathcal{I}}$ is a countable covering of $\Lambda_1$ by homeomorphic copies of $\R$, and each $R_i$ contains rays representing a positive end of $\Lambda_1$ and a negative end of $\Lambda_1$. Note that a priori, we could have $R_{i_1} = R_{i_2}$ for some $i_1 \neq i_2$, but we may take a subcover with only distinct copies of $\R$. With a slight abuse of notation, we will also use $\{R_i\}_{i \in \mathcal{I}}$ to denote this refined covering by distinct copies of $\R$. Similarly, we can perform an analogous covering by a countable collection of distinct copies of $\R$, $\left\{R_j\right\}_{j \in \mathcal{J}}$, for $\Lambda_2$. \par

%\item To get a left-order on $H$: cover $\Lambda_1$ with countably many copies of $\R$, where each copy $R_i$ has one end in $\Ends^+$ and one in $\Ends^-$. Since $H$ fixes ends, get action on each $R_i$ and a homomorphism $\varphi: H \to \prod_{i} \mathrm{Homeo}^+(R_i) \cong \prod_i \mathrm{Homeo}^+(\R)$. 
%Since the product is left-orderable, it's enough to show that $\varphi$ is injective. 

%an element of $\mathrm{Ends}_+(\Lambda_1)$ and another ray, also contained in $R_i$, corresponding to an element of $\mathrm{Ends}_-(\Lambda_1)$.

Now, let $\mathcal{K} = \mathcal{I} \sqcup \mathcal{J}$. For each $k \in \mathcal{K}$, both the positive end and the negative end represented by rays contained in $R_k$ are fixed by the action of the group $H$. With the choice of our subcovers, simple connectedness of $\Lambda_1, \Lambda_2$ implies that given any pair consisting of a positive end $x_1$ and a negative end $x_2$ of each $\Lambda_i$, there exists at most one $k_0$ for which $R_{k_0}$ contains rays representing $x_1$ and $x_2$. This means that each $R_k$ must be preserved by the action of $H$. Then for each $k \in \mathcal{K}$, we have a homomorphism $\varphi_k\colon H \to \mathrm{Homeo}_+(R_k)$ associated to the induced action of $H$ on $R_k$. Together, these homomorphisms give rise to a homomorphism $\varphi \colon H \to \prod_{k \in \mathcal{K}} \mathrm{Homeo}_+(R_k)$ by defining the $k$-th factor of $\varphi(h)$ to be $\varphi_k(h)$. The group $\prod_{k \in \mathcal{K}} \mathrm{Homeo}_+(R_k)  \cong  \prod_{k \in \mathcal{K}} \mathrm{Homeo}_+(\R)$ is left-orderable since it is a countable product of left-orderable groups (\cite{deroin2014groups}). Thus, to show the left-orderability of $H$ it is enough to show that $\varphi$ is injective.

%\item To show that $\varphi$ is injective: if $g \in \mathrm{Ker}(\varphi)$, then $g$ preserves each leaf in the bifoliated plane. This implies it fixes each end of a leaf in $\F_1$ in the boundary circle $\partial P$. If no infinite product regions, then $\mathrm{Ends}(\F_1)$ is dense in $\partial P$, so $g$ acts as the identity in the boundary circle. But then it must fix every point in $P$, but we were assuming the action of $G$ on $P$ is faithful, contradiction. So $\varphi$ is injective and then $H$ is orderable.
Let $h \in \mathrm{Ker}(\varphi)$. Then by definition, $h$ acts as the identity on $R_k$ for all $k$. Since $\{R_k\}_{k \in \mathcal{K}}$ covers $\Lambda_1 \cup \Lambda_2$, every point in each $\Lambda_i$ is fixed by $h$. This means that the leaves of the foliations $\F_1, \F_2$ must be preserved by the action of $h$ on $P$. Given any point $x \in P$, since both $\F_1(x)$ and $\F_2(x)$ are preserved by $h$, we must have their unique intersection fixed by $h$ as well, so $h(x) = x$ and we conclude that $h$ is the identity map on $P$. 

%Then $h$ acts as the identity on $P$, which means that $h$ must be the identity element of $H$ since the action of $G $ on $P$ was assumed to be faithful \todo{do we just say $h$ must be the identity since $G = \mathrm{Aut}^+(P)$?}. 

Therefore, $\mathrm{Ker}(\varphi) = \left\{\mathrm{id} \right\}$, so $\varphi$ is injective, hence $H$ is left-orderable. From the argument above it then follows that $G$ must be left-orderable.
\end{proof}

\section{Connection to Circle at Infinity}\label{section: circle at infinity}

Given a bifoliated plane $P = (P, \F_1, \F_2)$, the circle at infinity $\partial P$ compactifies $P$ to a closed disk $P\cup \partial P$, where the set of endpoints of leaves of the foliations $\F_1, \F_2$ form a dense subset of the \emph{circle at infinity} $\partial P$. 

%
%The circle at infinity associated to a bifoliated plane was defined initially by Fenley (\cite{fenley2012ideal}) for the case of the bifoliated plane $(P_\varphi, \F^+, \F^-)$ associated to an Anosov flow. Later, Bonatti in \cite{bonatti} gave a definition of the circle at infinity for a general bifoliated plane (in fact, for a plane with a countable number of transverse foliations or potentially singular foliations whose singular points are $k$-prongs for $k > 2$, but we will not need this here). 
%

In Section \ref{subsection: circle at infinity} we state the definition and the main properties of the circle at infinity for a bifoliated plane, following Bonatti's work in \cite{bonatti}. To streamline exposition, we use the notation due to Mather in \cite{mather}, where ideal boundaries are defined for certain foliations of surfaces via an essentially identical procedure. For details and proofs of the results we will state, see \cite{bonatti}. Theorem \ref{theorem: faithful} will be proved in Section \ref{subsection: realization of ends}.

\subsection{Boundary circle at infinity of a bifoliated plane}\label{subsection: circle at infinity}

The main idea in the construction of the circle at infinity  associated to a bifoliated plane $(P, \F_1, \F_2)$ is the following: given a foliation by lines $\mathcal{F}$ of the plane, let $\mathrm{Ends}(\mathcal{F})$ denote the set of all ends of leaves of $\mathcal{F}$, that is
\[
\mathrm{Ends}(\mathcal{F}) = \bigcup_{l \text{ leaf of } \mathcal{F}} \mathrm{Ends}(l).
\]

One can then give a canonical \emph{circular order} $\mathcal{O}$ to the set $\mathrm{Ends}(\F_1)\cup \mathrm{Ends}(\F_2)$ when leaves are properly embedded lines in $P$ and the foliations are transverse to each other. This is possible since in this case any two leaves (belonging to the same or different foliations) can intersect at most once, and they must leave any compact subset of $P$. Given any three ends $\{x_1, x_2, x_3\}$ of leaves of $\mathcal{F}_1$ or $\mathcal{F}_2$, one can find a simple closed curve that intersects the rays corresponding to these ends exactly once. Giving this curve the boundary orientation corresponding to the compact region it bounds allows one to declare the triple $(x_1, x_2, x_3)$ as either positively or negatively oriented, and this can be shown to be independent of the choice of simple closed curve.

\begin{rmk}
    Given a set $S$ equipped with a circular order $\mathcal{O}$, there are two operations one can perform on $(S, \mathcal{O})$ which yield circularly ordered sets related to $(S, \mathcal{O})$.
    \begin{enumerate}
        \item \emph{Identification of equivalent points}: we say two points $x, y\in S$ are equivalent if there are only finitely many points of $S$ between them on either side. The circularly ordered set $(\hat{S}, \hat{\mathcal{O}})$ is obtained by identifying all equivalent points. 

        \item \emph{Completion}: analogously to the Dedekind completion of a linearly ordered set, a circularly ordered set $(S, \mathcal{O})$ can be completed in a natural way in order to yield a complete circularly ordered set $(\Tilde{S}, \tilde{\mathcal{O}})$ that contains $(S,\mathcal{O})$. More precisely, Calegari showed in \cite{calegari_groups_order} that a circular order on a set $S$ consists of a collection of linear orders on all subsets of the form $S\setminus \{ x\}$, subject to some compatibility conditions. It can be seen that the usual Dedekind completions of these linear orders also satisfy the compatibility conditions, and yield a complete circular order on $S$.
    \end{enumerate}
\end{rmk}

%---------------------------- OLD SETUP FOR FIGURES
%\begin{figure}[h!]
%  \centering
%  \begin{minipage}[b]{0.45\textwidth}
%    \includegraphics[width=\textwidth]{reeb_plane.pdf}
 %   \caption{Part of a bifoliated plane consisting of two transverse Reeb foliations.}
 %   \label{fig:image8}
 % \end{minipage}
 % \hspace{2cm}
 % \begin{minipage}[b]{0.30\textwidth}
 %   \includegraphics[width=\textwidth]{reeb_circle.pdf}
 %   \caption{The compactification of the region shown on the left.}
 %   \label{fig:image9}
 % \end{minipage}
%\end{figure}
%----------------------------------
\begin{figure}[h!]
  \centering
\begin{minipage}[b]{0.9\textwidth}
  \includegraphics[width=\textwidth]{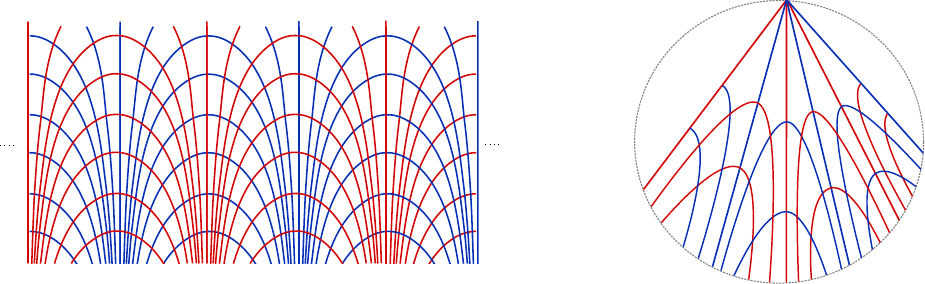}
    \caption{Part of a bifoliated plane consisting of two transverse Reeb foliations (left), and its compactification (right).}
    \label{fig:image8}
  \end{minipage}
 \end{figure}

This allows one to define the circle at infinity associated to a bifoliated plane:

\begin{definition}
    Let $(P, \F_1, \F_2)$ be a bifoliated plane, let $E =  \mathrm{Ends}(\F_1)\cup \mathrm{Ends}(\F_2) $ be the set of ends of leaves of $\F_1$ and $\F_2$, and let $\mathcal{O}$ be the natural circular order on $E$. Then, the \emph{circle at infinity} associated to $(P, \F_1, \F_2) $ is the circularly ordered set $\partial P =  \Tilde{\hat{E }  } $ with order $\Tilde{ \hat{ \mathcal{O} } }$. 
\end{definition}

Elements of $\partial P$ are called \emph{ideal points}. By construction, each end of a leaf corresponds to an ideal point.

\begin{theorem}[\cite{bonatti}]
      The set $P \cup \partial P$ admits a natural topology that makes it homeomorphic to a closed disk which is the compactification of the plane $P$ by the circle $\partial P \cong S^1$. This topology is such that any homeomorphism of the plane preserving the foliations $\F_1, \F_2$ induces a homeomorphism of $P\cup \partial P$. 
      
\end{theorem}

There is a natural map $E \to \partial P$, assigning an ideal point to each end of a leaf. In general, the image of this map will be dense in, but not equal to, $\partial P$. However, we can say more: except in the case of certain bifoliated planes containing \emph{infinite product regions}, the image of each of the sets $\mathrm{Ends}(\F_1), \mathrm{Ends}(\F_2) \subset E$ is dense in $\partial P$.

\begin{definition}
An \emph{infinite product region} in $(P, \F_1, \F_2)$ is an open subset $U \subset P$ such that for $i, j = 1, 2$ and $i \neq j$,
\begin{enumerate}
    \item the boundary of $U$ consists of a compact segment of an $\F_i$-leaf $l$ and two $\F_j$-leaves, $j = 3 - i$, $f_1$, $f_2$ that intersect $l$;
    \item for each $x \in U$, $\F_i(x)$ intersects both $f_1$ and $f_2$.
\end{enumerate}If $U$ is an infinite product region, we say that \emph{$U$ is based in $\F_i$}, $i = 1, 2$, if the compact segment of $\partial U$ inside $P$ is on an $\F_i$-leaf.
\end{definition}

\begin{proposition}
    Let $P = (P, \F_1, \F_2)$ be a bifoliated plane. For $i = 1, 2$, $P$ contains no infinite product region based in $\F_i$ if and only if $\mathrm{Ends}(\F_i)$ is dense in $\partial P$.
\end{proposition}

\begin{proof}
    We start with the \emph{if} direction. Suppose that $P$ contains an infinite product region $U$ based in $\F_i$. Let $I = \partial U \cap \partial P$, then by definition, $I \cap \Ends(\F_i) = \emptyset$. Thus $\Ends(\F_i)$ is not dense in $\partial P$. \par
    
    For the \emph{only if} direction, suppose that $P$ does not contain any infinite product region based in $\F_i$. We show that given any interval $I \subset \partial P$, there exists $x \in P$ such that $\partial \F_i(x) \cap I \neq \emptyset$. Suppose such $x$ does not exist. Pick two distinct $\F_j$-leaves $l_1, l_2$ such that each has an ideal point in $I$. Then there exists $x^\prime \in P$ such that $\F_i(x^\prime)$ either bounds an infinite product region together with $l_1$ and $l_2$ or intersects one of them twice. Neither is allowed here, so by contradiction we get the density of $\Ends (\F_i)$.
\end{proof}

So far in this section, we have made no mention of the orientations on the foliations $\F_1$ and $\F_2$ of a bifoliated plane. In the presence of such orientations, it will be useful to further divide elements of each $\mathrm{Ends}(\F_i)$ into positive and negative ideal points.
\begin{definition}
    Given a leaf $l$ of $\F_i$ equipped with the orientation induced from the foliation, we say an ideal point $x \in \mathrm{Ends}(\F_i)$ of $l$ is a \emph{positive ideal point of} $l$ if $x = \lim_{t \to +\infty} \gamma(t)$, where $\gamma \colon [0, +\infty) \to l$ is an orientation preserving and proper continuous map. 
    
    Otherwise (i.e.~if the same holds with $\gamma$ orientation reversing), we say $x$ is a \emph{negative ideal point} of $l$.

\end{definition}

We will also use the notion of \emph{quadrant} to describe relative positions of points.

\begin{definition}
    Let $x \in P = (P, \F_1, \F_2)$. We call each connected component of $P \setminus (\F_1(x) \cup \F_2(x))$ a \emph{quadrant} given by $x$. Let $\psi: P \to \R^2$ be an orientation-preserving local homeomorphism that sends $x$ to the origin, and sends $\F_1(x)$, $\F_2(x)$ to the oriented $x$-axis and $y$-axis, respectively. The pre-images of the first, second, third, and fourth quadrants in $\R^2$ under $\psi$ are said to be the \emph{top-left}, \emph{top-right}, \emph{bottom-left}, and \emph{bottom-right} quadrants of $x$, respectively.
\end{definition}

\subsection{Realizing ends on the circle at infinity}\label{subsection: realization of ends}

We build a natural correspondence between the set of ends of the leaf spaces of the bifoliation and certain subset of the boundary circle at infinity. We present the result here in the form of Proposition \ref{proposition: mapping to ends}. \par

To provide some intuition, recall that each point in the leaf space $\Lambda_i$ represents a leaf in the plane under the projection map $q_i$. Then the pre-image of each ray $r_i \in \Lambda_i$ under $q_i$ is an unbounded monotone sequence of $\F_i$-leaves $\{l_\alpha\}_{\alpha \in \mathcal{A}}$. We can then represent $\{l_\alpha\}_{\alpha \in \mathcal{A}}$ with a transverse ray $\tau \subset P$ that intersects every leaf in this sequence. Note that $\tau$ is not necessarily a half-leaf of $\F_j$, $j = 3 - i$. Moreover, $\tau$ cannot be contained in any compact subset of the plane, and we will show that its unbounded side must accumulate to a single point on the boundary circle at infinity. \par

We start the construction by a lemma that characterizes the limit set of a ray which has no accumulation inside a topological plane .
     
\begin{lemma}\label{lemma: curve connected limit set}
    Let $P$ be a topological plane. If $\tau \subset P$ is a ray with no accumulation inside $P$, then $ \partial^+ \tau = \lim_{t \to +\infty}\tau[t, +\infty)$ is a connected subset of $\partial P \cong \mathbb{S}^1$.
\end{lemma}

\begin{proof}
    Suppose that $\partial^+ \tau$ is disconnected in $\partial P$. Since $\tau$ is a connected ray, $\partial^+ \tau$ is connected in $P \cup \partial P$. Then $\partial^+ \tau \cap P \neq \emptyset$. But this contradicts that $\tau$ has no accumulation inside $P$.
\end{proof}

For each foliation $\F_i$, $i = 1, 2$, any interval $I$ in the leaf spaces $\Lambda_i$ is the image under the projection $q_i$ of some curve $\tau$ transverse to $\F_i$: given any point $x$ on a transverse curve $\tau$ to $\F_i$, $q_i(\F_i(x))$ a point in $\Lambda_i$, so $q_i(\tau)$ is connected in $\Lambda_i$ and contains more than one point, hence an interval. Since ends in each leaf space are equivalent classes of rays, it is only natural to study \emph{transverse rays} to a given foliation. \par

Unless otherwise mentioned, for the rest of this paper we assume transverse rays to be parameterized by $[t, +\infty)$, i.e.~embeddings of intervals of the form $[t, +\infty)$. 

Since we will be dealing with limits and accumulation sets frequently, to avoid excessive repeating of the term, we make the following convention by slightly abusing the notation $\partial$:

\begin{convention}
    \hfill
    \begin{enumerate}
        \item If $\tau \in P$ is a transverse ray to $\F_i$, then $\partial^+ \tau$ denotes the accumulation of $\tau$ in its unbounded direction;
        \item If $l \in \F_i$, then $\partial l$ denotes the set of (two) ideal points of $l$;
        %\item If $L \subset \F_i$ is a collection of accumulating leaves, then $\partial L$ denotes the set of accumulation points of ideal points of leaves in $L$.
    \end{enumerate}
\end{convention}

%In addition, for each $i = 1, 2$, if $\xi \in \partial P$ is an ideal point of some $\F_i$-leaf $l$, then we can extend the projection map to $q_i \colon P \cup \Ends(\F_i) \to \Lambda_i$ by setting $q_i(\xi)$

We describe some rules of how transverse curves in $P$ must behave. 

\begin{lemma}\label{lemma: transversal no double intersection}
    If $\tau$ is a curve in $P$ transverse to $\F_i$, then $\tau$ can only intersect an $\F_i$-leaf at most once. Moreover, if $\tau$ has an ideal point on $\partial P$ and intersects an $\F_i$-leaf $l$ in $P$, then $\tau$ and $l$ cannot share a common ideal point on $\partial P$.
\end{lemma}

\begin{proof}
    Suppose that $\tau$ intersects some $\F_i$-leaf $l$ at two distinct points $x_1$ and $x_2$. Then the transverse orientations on any sufficiently small neighbourhoods of $x_1$ and $x_2$ must be opposite. But this implies that there exists $x_3 \in \tau$ between $x_1$ and $x_2$ at which $\tau$ is tangent to $\F_i$, contradicting that $\tau$ is transverse to $\F_i$. \par

    Now suppose that $\tau \cap l = \tau(t_0)$ and $\partial^+ \tau \cap \partial l = \xi$. Since any $\F_i$-leaf can only intersect $\tau$ at most once and two distinct $\F_i$-leaves cannot intersect each other, for all $x \in \tau(t_0, +\infty)$, we have $\xi \in \partial \F_i(x)$. But $\{F_i(x)\}_{x \in \tau(t_0, +\infty)}$ is an uncountable collection of leaves, this is impossible by Lemma 3.2 in \cite{bonatti}.
\end{proof}

The next lemma shows that any transverse ray must land at some point on the circle at infinity.

\begin{lemma}\label{lemma: transverse ray lands on the boundary}
    If $\tau \subset P$ is a transverse ray to $\F_i$, then $\tau$ cannot accumulate inside $P$, meaning that there does not exist any $x \in P \setminus \tau$ such that $\tau$ intersects every neighbourhood of $x$. Furthermore, $\partial^+ \tau$ must be a point on $\partial P$.
\end{lemma}

\begin{proof}
    Suppose that there exists $x \in P \setminus \tau$ that is accumulated by $\tau$. Take a small neighborhood $U$ of $x$, which is foliated by leaves of $\F_i$. Then $\tau$ necessarily intersects $\F_i(x)$ infinity many times as it is a transverse ray, which contradicts Lemma \ref{lemma: transversal no double intersection}. \par
    By Lemma \ref{lemma: curve connected limit set}, $\partial^+ \tau$ is a connected subset of $\partial P \simeq \mathbb{S}^1$ -- so it is either the entire $\mathbb{S}^1$, an interval or a single point. As a first case, suppose that $\partial^+ \tau = \mathbb{S}^1$. Pick a leaf $l \in \F_i$, which separates $P$ into two half-planes $H_1$ and $H_2$, and assume that $\tau(0) \in H_1$. Then there exists $t_1, t_2, t_3 \in \R$ such that $t_1 < t_2 < t_3$ and $\tau[0, t_1) \subset H_1, \tau(t_1, t_2) \subset H_2, \tau(t_2, t_3) \subset H_1$. But this implies that $\tau(t_1) \in l$ and $\tau(t_2) \in l$, contradicting Lemma \ref{lemma: transversal no double intersection}. For the second case, suppose that $\partial^+ \tau$ is some interval $I = [\gamma_1, \gamma_2] \subset \partial P$. Pick a leaf $l^\prime \in \F_i$ with at least one ideal point in $\mathring{I}$ and let $\xi$ be such ideal point. Then $\tau \cap l^\prime \neq \emptyset$. Let $z = \tau(t_0) \in \tau \cap l$. By Lemma \ref{lemma: transversal no double intersection}, $\tau(t_0, +\infty)$ is contained in exactly one of the half-planes separated by $l^\prime$, so we have either $\partial^+ \tau \subset [\gamma_1, \xi]$ or $\partial^+ \tau \subset [\xi, \gamma_2]$. But neither satisfies that $\partial^+ \tau = I$, so we have a contradiction. Therefore, we can only be in the third case, that is, $\partial^+ \tau$ must be a point on $\partial P$.
\end{proof}

It follows that every transverse ray has an ideal point. For this, we say that a transverse ray \emph{lands at} some point on the boundary at infinity. \par

Before stating and proving our main proposition, we build the subset on the boundary circle at infinity whose elements represent ends of leaf spaces. Intuitively, given a transverse ray $\tau$ to $\F_i$, $\tau$ landing at some point on $\partial P$ is a necessary condition for $q_i(\tau)$ to be a ray in $\Lambda_i$. But it is not sufficient. We first describe configurations on $\partial P$ that prevents us from getting rays in $\Lambda_i$.

\begin{lemma}\label{lemma: R_i cannot be rays}
    Let $\xi \in \partial P$ be an ideal point of some $\F_i$-leaf or an accumulation point of ideal points of non-separated $\F_i$-leaves, and let $\tau \subset P$ be a transverse ray to $\F_i$ landing at $\xi$. Then $q_i(\tau)$ is bounded in $\Lambda_i$. In particular, since $q_i(\tau)$ is connected, it must be an interval and thus cannot be a ray. 
\end{lemma}

\begin{proof}
    Let $x \in \tau$. If there exists some $l \in \F_i$ such that $l$ has an ideal point at $\xi$, then $q_i(\F_i(x))$ is bounded from above by $q_i(l)$. If there exists a collection of non-separated leaves $L \subset \F_i$ such that its elements have ideal points accumulating to $\xi$, then $q_i(\F_i(x))$ is bounded from above by any point contained in the cataclysm $q_i(L)$. Since this is true for any $x$, we conclude that $q_i(\tau)$ is bounded.
\end{proof}

%\begin{definition}
%    Let $l_1, l_2 \in \F_i$ be two non-separated leaves with $l_1 < l_2$. If they do not share any common ideal point, and if there is no $l_3 \in \F_i$ non-separated from both $l_1$ and $l_2$ such that $l_1 < l_3 < l_2$, then we call the interval bounded by the positive ideal point of $l_1$ and the negative ideal point of $l_2$ the \emph{gap} between $l_1$ and $l_2$. We also say that it is a gap in $\F_i$.
%\end{definition}

\begin{definition}
    Let $\{l_n\}_{n \in \N} \subset \F_i$ be a convergent sequence of leaves such that $\lim_{n \to \infty} l_n$ contains at least one $\F_i$-leaf. A \emph{gap} in $\F_i$ is a connected component of $\lim_{n \to \infty} l_n \cap \partial P$ that has non-empty interior in $\partial P$.
\end{definition}

\begin{remark}
    If $\mathcal{G}$ is a gap in $\F_i$, then by definition, $\mathring{\mathcal{G}}$ does not contain any ideal point of $\F_i$-leaves. Moreover, at least one endpoint of $\partial \mathcal{G}$ is an ideal point of some $\F_i$-leaf or an accumulation point of ideal points of non-separated $\F_i$-leaves, otherwise $\lim_{n \to \infty} l_n$ cannot contain any $\F_i$-leaf.
\end{remark}

Note that with this definition, for each $\F_i$, there are two types of infinite product regions: ones whose boundary component in $\partial P$ is contained in a gap in $\F_i$, and ones whose boundary component in $\partial P$ is not part of any gap. Since these types will be crucial to our construction, we make the definition formal here:

\begin{definition}
     Let $U$ be an infinite product region based in $\F_i$, and let $I = \partial U \cap \partial P$. If $I$ is contained in some gap in $\F_i$, then $U$ is said to be a \emph{gap product region based in $\F_i$}. Otherwise, $U$ is said to be a \emph{standard product region based in $\F_i$}.
\end{definition}

\begin{figure}[h]
    \centering
    \includegraphics[scale = 0.85]{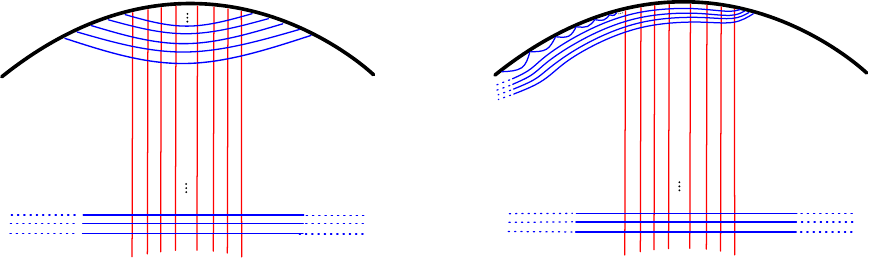}
    \caption{On the left is a standard product region based in the \emph{blue} foliation, and on the right is an example of a gap product region based in the \emph{blue} foliation with accumulating non-separated \emph{blue} leaves on one side.}
    \label{fig: types of product regions}
\end{figure}

\begin{lemma}\label{lemma: ray cannot land in a gap}
    The projection of any transverse ray to $\F_i$ landing in a gap in $\F_i$ is bounded $\Lambda_i$.
\end{lemma}

\begin{proof}
    Let $\mathcal{G}$ be a gap in $\F_i$. For any two transverse rays $\tau_1, \tau_2$ to $\F_i$ that land in $\mathcal{G}$, there exists $x_1 \in \tau_1$ and $x_2 \in \tau_2$ such that $\F_i(x_1) = \F_i(x_2)$. Then with reparametrization, there exists $t_0 \in (0, +\infty)$ such that $\tau_1[t_0, +\infty) = \tau_2[t_0, +\infty)$, so $q_i(\tau_1)$ and $q_i(\tau_2)$ are either both unbounded or both bounded. \par
    
    To prove the lemma, first, let $\tau \subset P$ be a transverse ray to $\F_i$ landing on some point in $\partial \mathcal{G}$ that is an ideal point of some $\F_i$-leaf or an accumulation point of ideal points of non-separated $\F_i$-leaves. Then $q_i(\tau)$ is bounded by Lemma \ref{lemma: R_i cannot be rays}. Now, if $\tau^\prime \subset P$ is a transverse ray to $\F_i$ that lands at any point in $\mathcal{G}$, then by the first paragraph of this proof, $q_i(\tau^\prime)$ must also be bounded.
\end{proof}

For $i = 1, 2$, let
\begin{align*}
    E_i = \{\xi \in \partial P \mid \xi \textrm{ is an ideal point of some transverse ray to } \F_i\}, 
\end{align*}
let
\begin{align*}
    R_i = \{\eta \in \partial P \mid \eta &\textrm{ is neither an ideal point of any } \F_i \textrm{ leaf} \\ &\textrm{ nor accumulation of ideal points of non-separated } \F_i \textrm{-leaves}\},
\end{align*}
and let
\begin{align*}
    N_i = \partial P \backslash \{\textrm{gaps in $\F_i$}\}.
\end{align*}
 
Let $S_i = E_i \cap R_i \cap N_i$, then we have the following results:

\begin{lemma}\label{lemma: transversals are rays}
    The projection of any transverse ray landing at some point in $S_i$ is a ray in $\Lambda_i$.
\end{lemma}

\begin{proof}
    Let $\xi \in S_i$ and let $\tau \subset P$ be a transverse ray to $\F_i$ that lands at $\xi$ with an initial point $x_0 \in P$. We show that $q_i(\tau)$ is unbounded in $\Lambda_i$, then the statement follows since $q_i(\tau)$ is connected. Suppose otherwise, and let $a = q_i(x_0) \in \Lambda_i$. Then there exists $b \in \Lambda_i$ such that $q_i(\tau) = [a, b)$. Then the pre-image of $b$ under $q_i$ is an $\F_i$-leaf that either has an ideal point at $\xi$ or is contained in some non-separated collection of $F_i$-leaves whose ideal points accumulate to $\xi$. But either of these means that $\xi \notin R_i$, and we have a contradiction.
\end{proof}

\begin{figure}[h]
    \labellist
    \small \hair 2pt
    \pinlabel $U$ at 122 56
    \pinlabel $\partial P$ at -13 81
    \endlabellist
    \centering
    \includegraphics[scale = 0.75]  {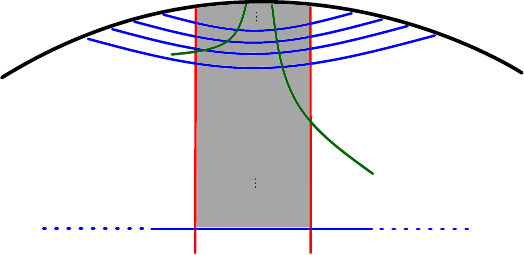}
    \caption{An end realized by the boundary of an infinite product region on $\partial P$}
    \label{fig: ends on boundary}
\end{figure}

\begin{lemma}\label{lemma: transversals landing at the same point represent an end}     
    If $\tau_1$ and $\tau_2$ are two distinct transverse rays to $\F_i$ that land at the same ideal point $\xi \in S_i$, then their projections are rays in the same equivalent class in $\Lambda_i$, i.e.~they represent the same end in $\Lambda_i$. Moreover, if $P$ contains a standard product region $U$ based in $\F_i$, then all transverse rays landing in $I = \partial U \cap \partial P$ represent the same end in $\Lambda_i$.
\end{lemma} 

\begin{proof}
    Let $\tau_1$ and $\tau_2$ be two transverse rays to $\F_i$, both landing at $\xi \in S_i$, with different initial points $x_1$ and $x_2$, and let $r_1 = q_i(\tau_1)$ and $r_2 = q_i(\tau_2)$ be their corresponding rays in $\Lambda_i$. We breakdown the proof of the first statement into 2 cases. \par
    
    \noindent\textbf{Case 1:} If either $F_i(x_1) \cap \tau_2 \neq \emptyset$ or $\F_i(x_2) \cap \tau_1 \neq \emptyset$, let $s$ be a point of such intersection, then $r_1$ and $r_2$ agree after passing $q_i(s)$, so $r_1$ and $r_2$ represent the same end. \par

    \noindent\textbf{Case 2:} Suppose that $\F_i(x_1) \cap \tau_2 = \emptyset$ and  $\F_i(x_2) \cap \tau_1 = \emptyset$. Without loss of generality, let $V$ be the region bounded by $\tau_1, \tau_2$, the positive half-leaf of $\F_i(x_1)$, the negative half-leaf of $\F_i(x_2)$ and $\partial P$ (Figure \ref{fig: transversals landing at the same point}). Then either there exists $x_0 \in \tau_1$ such that the positive half-leaf of $\F_(x_0)$ intersects $\tau_2$, or for all $x \in \tau_1$ the positive half-leaf of $\F_(x)$ has an ideal point in $\partial V \cap \partial P$. Suppose we have the first case. With a slight abuse of notation, by considering $\tau_1$ with the initial point $x_0$ instead of $x_1$, we are back in \textbf{Case 1}. Now suppose we have the second case. Then for all $y \in \tau_2$, the negative half-leaf of $\F_i(y)$ cannot intersect $\tau_1$ and thus must have an ideal point in $\partial V \cap \partial P$. Then as $x$ and $y$ converge to $\xi$, $\F_i(x)$ and $\F_i(y)$ must converge to a single $F_i$-leaf that has an ideal point at $\xi$, contradicting $\xi \in S_i$. Since the second case cannot happen, we have that $[r_1] = [r_2]$. \par

    Now, suppose that $P$ contains a standard product region $U$ based in $\F_i$. Let $I = \partial U \cap \partial P$. By definition, $I$ is not a gap in $\F_i$, and it does not contain any point that is an ideal point of a $\F_i$-leaf. Moreover, for any $\xi \in I$, there exists a unique $\F_j$-leaf, $j = 3 - i$, $f$ such that $f \cap U \neq \emptyset$ and $\xi \in \partial f$. Hence, $I \subset S_i$, and every point in $I$ represents an end in $\Lambda_i$ by Lemma \ref{lemma: transversals are rays}. Since any two $\F_j$-leaves landing in $I$ intersect the same set of $\F_i$-leaves after entering $U$, all $\F_j$-leaves landing in $I$ represent the same equivalent class of rays in $\Lambda_i$. Since any transverse ray landing in $I$ shares an ideal point with some $\F_j$-leaf, by the first statement of this lemma, which we have proved, all transverse rays landing in $I$ represent the same end in $\Lambda_i$.
\end{proof}

\begin{figure}[h]
    \labellist
    \small \hair 2pt
    \pinlabel $V$ at 152 76
    \pinlabel $x_1$ at 75 78
    \pinlabel $x_2$ at 228 87
    \pinlabel $\tau_1$ at 118 114
    \pinlabel $\tau_2$ at 199 114
    \pinlabel $\xi$ at 157 176
    \endlabellist
    \centering
    \vspace{3mm}\includegraphics[scale = 0.65]{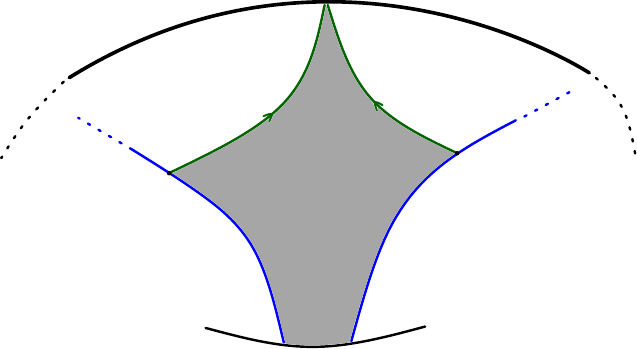}
    \caption{Two transversals landing at the same point}
    \label{fig: transversals landing at the same point}
\end{figure}
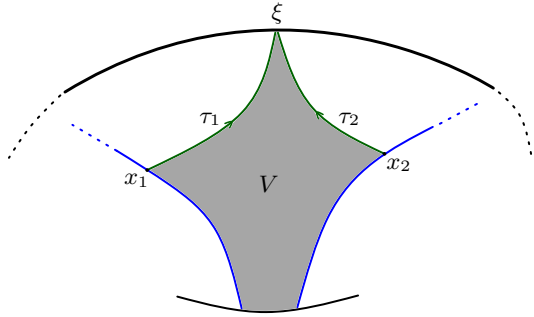

\begin{corollary}\label{corollary: increasing product regions}
    If $\{I_\alpha\}_{\alpha \in \mathcal{A}} \subset \partial P$ is an increasing (for inclusion) sequence of intervals such that each $I_\alpha$ is the boundary component of a standard product region $U_\alpha$ based in $\F_i$, then $I = \bigcup_{\alpha \in \mathcal{A}} I_\alpha$ represents the same end in $\Lambda_i$ as any $I_\alpha$.
\end{corollary}

\begin{proof}
    Let $\xi, \xi^\prime \in \partial P$ so that $I = [\xi, \xi^\prime]$. If $I_\alpha \supset I_\beta$, then $I_\alpha$ and $I_\beta$ represent the same end by Lemma \ref{lemma: transversals landing at the same point represent an end}. So if $I$ is a boundary component of a standard product region, then $\xi$ and $\xi^\prime$ are ideal points of $\F_j$-leaves, $j = 3 - i$, and the statement follows since $I$ contains all $I_\alpha$, $\alpha \in \mathcal{A}$. \par
    
    Now suppose that $I$ does not bound any standard product region. Then $\xi$ and $\xi^\prime$ are accumulated by ideal points of $\F_j$-leaves that bound standard product regions in $\{U_\alpha\}_{\alpha \in \mathcal{A}}$. We need to show that both $\xi$ and $\xi^\prime$ are elements of $S_i$. Then we get that $I \subset S_i$, and the proof will be completed by continuity and Lemma \ref{lemma: transversals landing at the same point represent an end}. \par
    
    In fact, by symmetry, we will only need to show that $\xi$ is a point in $S_i$. Let $U = \bigcup_{\alpha \in \mathcal{A}} U_\alpha$. Let $x \in U$. Moreover, without loss of generality, fix a transverse orientation so that $\xi$ is in the top-left quadrant of $x$. Pick a sequence of points $\{x_n\}_{n \in \N} \subset U$ that converge to $\xi$ such that $x_1 = x$ and $x_{n+1}$ is in the top-left quadrant of $x_n$ for all $n \in \N$. Then we get a transverse ray to $\F_i$ that lands at $\xi$ by connecting $x_n$ and $x_{n+1}$ with a straight line for all $n \in \N$. Hence, $\xi \in E_i$. By construction, $\xi$ must be in $R_i$ because otherwise it implies that $I$ is a gap in $\F_i$, but $I$ is the union of boundary components of standard product regions. For the same reason, $\xi$ must be in $N_i$. Therefore, $\xi \in S_i$, and we are done.
    %Observe that given the way $\xi$ was constructed, it cannot be accumulated by ideal points of non-separated $\F_i$-leaves since there is no non-separated $\F_i$-leaves in a small neighborhood of $\xi$, so we only need to show that $\xi$ is not an ideal point of any $\F_i$-leaf. Suppose that $\xi$ is an ideal point of some $\F_i$-leaf $l$. Then $l \cap U \neq \emptyset$. Take $x \in U$ that also lies in the positive half-plane defined by $l$. Then $\partial \F_i(x) \cap (\xi, \xi^\prime) \neq \emptyset$, but this contradicts that any two distinct $\F_j$-leaves landing in $(\xi, \xi^\prime)$ define an standard product region based in $\F_i$, and we are done.
\end{proof}

We are now ready to complete the construction. Since a group action on $P$, preserving foliations and their orientations, induces an action on $\partial P$, with the following proposition we will have the tool to associate group actions on $P$ with group actions on the set of ends of the leaf spaces --- this is the crucial step in proving Theorem \ref{theorem: faithful}.

\begin{proposition}\label{proposition: mapping to ends}
    Let $\xi \in S_i$. For $i = 1, 2$, let $\tau_\xi$ be any transverse ray to $\F_i$ that has an ideal point at $\xi$, and let $r_\xi = q_i(\tau_\xi) \in \Lambda_i$. Then the map $\Phi_i \colon S_i \to \Ends(\Lambda_i)$ such that $\Phi_i(\xi) = [r_\xi]$ is well-defined and surjective, and the pre-image of any end in $\Ends(\Lambda_i)$ is either a single point, or a closed interval $I \in \partial P$ such that any two $\F_j$-leaves, $j = 3 - i$, ending in $\mathring{I}$ bound a standard product region based in $\F_i$.
    
    %$\partial I$ consists of either ideal points of $\F_j$-leaves that bound a standard infinite product region based in $\F_i$ or accumulation of such ideal points.
\end{proposition}

\begin{proof}
    The map $\Phi_i$ is well-defined by Lemma \ref{lemma: transversals are rays} and Lemma \ref{lemma: transversals landing at the same point represent an end}. \par
    
    We first check the surjectivity of $\Phi_i$. Let $[r] \in \Ends(\Lambda_i)$. Then in $\Lambda_i$, $r$ is a ray with some initial point $x_0$. The pre-image of $r$ under $q_i$ in $P$ is saturated by a collection of $\F_i$-leaves, $\{l_\alpha\}_{\alpha \in \mathcal{A}}$, with an ``initial'' leaf $l_{\alpha_0}$ whose projection in $\F_i$ is $x_0$. We can find some transverse ray $\tau \subset P$ to $\F_i$ that starts at some point on $l_{\alpha_0}$ and intersects every leaf of $\{l_\alpha\}_{\alpha \in \mathcal{A}}$ so that $q_i(\tau) = r$. By Lemma \ref{lemma: transverse ray lands on the boundary}, it lands at some point $\eta \in \partial P$. By Lemma \ref{lemma: ray cannot land in a gap}, we know that $\eta$ is not contained in any gap. We are left to check that $\eta \in R_i$, and this is true by Lemma \ref{lemma: R_i cannot be rays}. \par

    To prove the last part of the proposition, we show that if $[r]$ is an end of $\Lambda_i$ such that $\Phi_i^{-1}([r])$ contains distinct points $\xi, \xi^\prime$, then there exists an interval $I \in \partial P$ such that $\xi, \xi^\prime \in I$ and $\Phi_i^{-1}([r]) = I$. Let $\tau$ and $\tau^\prime$ be transverse rays to $\F_i$ that lands at $\xi$ and $\xi^\prime$, respectively. Then by Lemma \ref{lemma: ray cannot land in a gap}, $\xi$ and $\xi^\prime$ are not in the boundary of any gap product region based in $\F_i$. Let $r_\xi = q_i(\tau_\xi)$ and $r_{\xi^\prime} = q_i(\tau_{\xi^\prime})$. Since $r_\xi \sim r_{\xi^\prime}$, then up to switching $r_\xi$ and $r_{\xi^\prime}$ there exists $t_0 \in \mathbb{R}$ such that $r_\xi[t_0, +\infty) = r_{\xi^\prime}[0, +\infty)$. Let $l_0 \subset P$ be the pre-image of $r_\xi(t_0)$ under $q_i$. Then $l_0$, as well as every $\F_i$-leaf whose projection under $q_i$ is contained in $r_\xi(t_0, +\infty)$, intersect both $\tau$ and $\tau^\prime$. If both $\tau$ and $\tau^\prime$ are $F_j$-leaves, then $l_0, \tau$, and $\tau^\prime$ form a standard product region whose boundary component in $\partial P$ is $J = [\xi, \xi^\prime]$, and $J \subset \Phi_i^{-1}([r])$ by Lemma \ref{lemma: transversals landing at the same point represent an end}. Otherwise, since ideal points of leaves are dense in $\partial P$, we can find sequences of $\F_j$-leaves, $\{f_n\}_{n\in \N}$ and $\{f^\prime_n\}_{n\in \N}$, landing in $J$ whose ideal points converge to $\xi$ and $\xi^\prime$, respectively. Since $r_\xi[t_0, +\infty) = r_{\xi^\prime}[0, +\infty)$, then for each $n$, we can find some $l_n \in F_i$ with $q_i(l_n) \in r_\xi(t_0, +\infty)$ such that $f_n$, $f^\prime_n$, and $l_n$ form a standard product region whose boundary component in $\partial P$ is the closed interval $J_n$. By construction, $\{J_n\}_{n \in \N}$ is an increasing (for inclusion) sequence of intervals, thus $J = [\xi, \xi^\prime] = \bigcup_{n \in \N} J_n \subset \pi^{-1}_i([r])$ by Corollary \ref{corollary: increasing product regions}. Following from this, if there exists $\xi^\star$ outside of $J$ such that $\xi^\star \in \Phi_i^{-1}([r])$, then up to switching $\xi$ and $\xi^\prime$, we have $[\xi^\star, \xi^\prime] \subset \Phi_i^{-1}([r])$. In particular, if $\Phi_i^{-1}([r])$ contains any interval $J$, then $\Phi_i^{-1}([r])$ is equal to some maximal interval $I$, in the sense of inclusion, that contains $J$. 
\end{proof}

\begin{comment}
    Since transversals landing at a common ideal point correspond to the same equivalence class of rays in the leaf space, it is enough to check the following: if $\xi \neq \xi^\prime \in \partial P$, $\tau$ and $\tau^\prime$ are transversals to $\F_i$ that land at $\xi$ and $\xi^\prime$, respectively, then $\pi(\tau) \nsim \pi(\tau^\prime) \subset \Lambda_i$. Let $r = \pi(\tau)$ and $r^\prime = \pi(\tau^\prime)$. Suppose that we have $r \sim r^\prime$, then up to switching $r$ and $r^\prime$, there exists $t_0 \in \mathbb{R}$ such that $r[t_0, +\infty) = r^\prime[0, +\infty)$. Let $l_0 \subset P$ be the image of $r(t_0)$, then every $\F_i$ leaf on the positive side of $l_0$ intersects both $\tau$ and $\tau^\prime$, so $l_0, \tau$, and $\tau^\prime$ form an infinite product region, which contradicts our assumption.
\end{comment}

Thus, we make for the following definition:

\begin{definition}
    Let $S = S_1 \cup S_2$. We call elements of $S$ the \emph{realizations} of ends of leaf spaces on $\partial P$.
\end{definition}

It follows from Proposition \ref{proposition: mapping to ends} that there are two \emph{types} of realizations of ends: a single point or an interval. For example, the leaf spaces of a trivial bifoliated plane, which are both homeomorphic to $\R$, have a total of four ends, which are realized by four intervals each consisting of endpoints of leaves along with their accumulation points. \par

Recall the setting of Theorem \ref{theorem: faithful}: $G$ is a group that acts on $P$ by homeomorphisms, preserving foliations and their orientations. We present its proof here:

\begin{proof}[Proof of Theorem \ref{theorem: faithful}]
    If $G$ acts faithfully on $P$, then for every non-identity element $g \in G$, there exists $x \in P$ such that $g(x) \neq x$. Then at least one of $\F_1(x)$ and $\F_2(x)$ is not fixed by $g$. The leaf that is not fixed by $g$ gives at least one ideal point on $\partial P$ that is not fixed by $g$. So $g$ does not act like the identity on $\partial P$, hence $G$ acts faithfully on $\partial P$. Conversely, if $G$ acts faithfully on $\partial P$, then for every non-identity element $h \in G$, there exists $\xi \in \partial P$ such that $h(\xi) \neq \xi$. Since any point on $\partial P$ is either an ideal point of some leaf or accumulated by ideal points of leaves, by continuity of the action, there exists some leaf that is not preserved by $h$, hence the action is faithful in $P$. \par

    The second \emph{if and only if} statement follows from the first plus the fact that point-type realization is dense in $\partial P$ implies that there is no infinite product region and hence there is a bijection between a dense subset of $\partial P$ and $\Ends(\Lambda_1, \Lambda_2)$ by Proposition \ref{proposition: mapping to ends}.
\end{proof}

Note that the induced action on $\partial P$ by $G$ preserves the induced orientation on $\partial P$.

\begin{corollary}
    Let $G$ be a group acting faithfully on $P$, preserving foliations and their orientations. If $G$ acts minimally on $\partial P$, then $G$ is left-orderable.
\end{corollary}

\begin{proof}
    Since $G$ acts minimally on $\partial P$, the orbit of any point-type realization of an end is dense in $\partial P$, implying that the set of point-type realization of ends is dense in $\partial P$. Then by Theorem \ref{theorem: faithful}, we get a faithful action on ends and thus obtain left-orderability of $G$. 
\end{proof}

An example of such minimal action is the Anosov-like action defined in \cite{barthelme2022orbit} on any non-trivial or skewed bifoliated plane. \par

The last result of this paper gives a more straightforward way of obtaining left-orderability of a group when the structure of leaf spaces are simpler, namely when there are only finitely many ends in some leaf space.

\begin{corollary}\label{corollary: lift to left-orderability}
If $G$ acts faithfully on $P$ preserving foliations and their orientations, and if $\Ends_+(\Lambda_i)$ or $\Ends_-(\Lambda_i)$ is finite for some $i$, then $G$ has global fixed point(s) on $\partial P$. In particular, $G$ is left-orderable.
\end{corollary}

Before proving this corollary, we make the following observation:

\begin{lemma}\label{lemma: no rotation}
    For any non-identity $g \in G$, the action on $\partial P$ cannot be conjugate to a non-trivial rotation.
\end{lemma}

\begin{proof}
    Suppose that some non-identity $g \in G$ acts by a non-trivial rotation on $\partial P$. Brouwer's Fixed Point Theorem says that $g$ has a fixed point $x \in P \cup \partial P$, but we know from our assumption that $x \notin \partial P$, so $x \in P$. Since foliations are non-singular and the group action preserves foliations, either $g$ preserves all 4 half-leaves of both $\F_1(x)$ and $\F_2(x)$, or $g$ switches positive and negative half-leaves of $\F_1(x)$ and does the same for $\F_2(x)$. In the former case, the action must be by the identity, and the latter reverses transverse orientation, so neither is possible.
\end{proof}

A classical criterion for left-orderability which we will make use of is the following: If a group acts faithfully on $\R$ by orientation preserving homeomorphisms, then it is left-orderable. A proof for this can be found in \cite{ghys_1984}. \par

Now we are ready for the proof:

\begin{proof}[Proof of Corollary \ref{corollary: lift to left-orderability}]
    First of all, since the induced action by $G$ on $\Lambda_i$ takes a positive (\emph{resp.} negative) end to another positive (\emph{resp.} negative) end, and since $G$ preserves foliations, a realization of an end can only be sent to another of the same type. In addition, recall that by Theorem \ref{theorem: faithful}, the induced action by $G$ on $\partial P$ is faithful. We will first show that there always exists at least one global fixed point under this action. \par

    If $\Lambda_i$ has at most one end with either the point-type or the interval-type realization, then the realizations of these ends are globally preserved by the action of $G$. If it is a point-type realization, then the point is a global fixed point of $G$; if it is an interval-type realization, then both endpoints of this interval are global fixed points of $G$. Now, if $\Lambda_i$ has more than one ends with the same given type of realization, then realizations of these ends are either all globally preserved by the action of $G$ -- in which case we get desired global fixed points as before, or in the same finite orbit by some non-identity element $g \in G$. In the later case, if realizations are of the interval-type, then negative (or positive) endpoints of these intervals are in the same finite orbit by $g$, and if realizations are of the point-type, then all of these points are in the same finite orbit by $g$. It follows that the action of $g$ is conjugate to a non-trivial rotation, contradicting Lemma \ref{lemma: no rotation}. This shows that $G$ has a global fixed point on $\partial P$. \par

    Now, pick a global fixed point $\xi$ of $G$ on $\partial P$. Then $\partial P \setminus \{\xi\} \cong \R$, so the faithful action by $G$ on $\partial P$ induces a faithful action on $\R$, and this implies that $G$ is left-orderable, see Theorem 6.8 in \cite{ghys_1984}.
\end{proof}   

\begin{comment}
    Suppose that  If the realization is a single point $\xi \in \partial P$, remove $\xi$ from $\partial P \cong \mathbb{S}^1$ -- this preserves the faithfulness of the action, then lift the faithful action by $G$ on $\partial P \backslash \{\xi\}$ to a faithful action on $\mathbb{R}$.  If the realization is a closed interval, left-orderability is obtained by removing one of its endpoints instead and lifting the faithful action to $\mathbb{R}$. \par

    Suppose that there are two ends in $\Lambda_i$. If their realizations are of different types, then both are globally fixed by $G$. From here, we can choose either realization and perform steps in the single-end case to get left-orderability of $G$. So suppose further that these ends have the same type of realization on $\partial P$. Then either they are both globally fixed by $G$, or the action by some elements in $G$ on $\partial P$ has a period-2 orbit and hence is conjugate to a rotation by $\pi$ on $\mathbb{S}^1$. The latter is impossible by the preceding lemma. So both are globally fixed by $G$. Then we choose either one and follow the steps in single-end case to get left-orderability of $G$. \par

    Finally, suppose that there are more than two ends in $\Lambda_i$. Then we must have at least two ends with the same type of realization, which are all globally fixed, and we obtain left-orderability as before.
\end{comment}

\section{Further Questions}

\subsection{Bifoliated plane with singular foliation}
In this paper, foliations of the plane were taken to be non-singular. So it is natural to ask if we can generalize our result for foliations with ($p$-prong type, $p \geq 3$) singularities. One way to obtain such bifoliated planes is by looking at orbit spaces of pseudo-Anosov flows on $3$-manifolds, following the same process as in the Anosov flow case. \par

The first obstacle in this direction is to build a linear order on the set of ends. Recall that we extended Zhao's linear order on positive ends to a linear order on all ends. But if there exists $p$-prong for some odd number $p$ in a foliation $\F_i$ , then $\F_i$ is not orientable, which makes one unable to define positive or negative ends. So a priori, a linear order will need to be defined directly on the set of all ends. It would be interesting to know if Zhao's construction can be modified to fit this requirement or a different approach is needed.

\subsection{Left-orderability of homeomorphism group of the disc}
In \cite{deroin2014groups}, the authors asked if the group of boundary-fixing and orientation-preserving homeomorphisms of the disc are left-orderable. A negative answer was given in \cite{hyde2019group}. However, Theorem \ref{theorem: main} shows that, with an extension to $P \cup \partial P \cong \mathbb{D}^2$, if the group preserves a bifoliation of the open disk together with its orientations, then it is left-orderable. One may ask the following question: what type of structures on a plane needs to be preserved by the action so that the group can be left-ordered? And conversely, if a group is left-orderable, what type of structures on a plane must the group action preserve?

\subsection{Homeomorphism group of the real line}
Let $G < \mathrm{Homeo}^+(\R)$. Then there is a trivial way of producing an action by $G$ on some bifoliated plane: Let $G$ act on two copies of $\R$, each representing the leaf space of a foliation, then we get an action by $G$ on the trivial bifoliated plane. We are interested in knowing if the converse is true, i.e.~if $G$ acts on some bifoliated plane, preserving foliations and their orientations, is $G$ isomorphic to a subgroup of $\mathrm{Homeo}^+(\R)$? Note that with Theorem \ref{theorem: main}, this is true for countable groups acting in the said ways. But it is not clear if we can replace countability with some weaker condition or drop it altogether. 

In case the answer to this question is negative, it would provide an example of a subgroup of $\mathrm{Homeo^+}(S^1)$ which is left-orderable but does not act (faithfully) on the line. At the same time, it would raise the question of which (uncountable) left-orderable subgroups of $\mathrm{Homeo}^+(S^1)$ have actions on bifoliated planes.

\printbibliography

\end{document}